\documentclass{amsart}
\usepackage{amssymb}
\copyrightinfo{2013}{V.X. Genest \emph{et al.}}

\newtheorem{theorem}{Theorem}
\newtheorem{proposition}{Proposition}
\newtheorem{lemma}{Lemma}

\theoremstyle{definition}

\theoremstyle{remark}
\newtheorem{remark}[theorem]{Remark}

\renewcommand{\tan}{\mathrm{tg}\,}

\newcommand{\under}[1]{_{#1}}%For references section%
\begin{document}

\title{A Laplace-Dunkl equation on $S^2$ and the Bannai--Ito algebra}
\author[V.X. Genest]{Vincent X. genest}
\address{Centre de Recherches Math\'ematiques, Universit\'e de Montr\'eal, P.O. Box 6128, Centre-ville Station, Montr\'eal (Qu\'ebec) H3C 3J7}
\email{genestvi@crm.umontreal.ca}
\author[L. Vinet]{Luc Vinet}
\address{Centre de Recherches Math\'ematiques, Universit\'e de Montr\'eal, P.O. Box 6128, Centre-ville Station, Montr\'eal (Qu\'ebec) H3C 3J7}
\email{vinetl@crm.umontreal.ca}
\author[A. Zhedanov]{Alexei Zhedanov}
\address{Donetsk Institute for Physics and Technology, Donetsk 83114, Ukraine}
\email{zhedanov@kinetic.ac.donetsk.ua}
\date{}
\dedicatory{}

%    Abstract is required.
\begin{abstract}
The analysis of the $\mathbb{Z}_2^{3}$ Laplace-Dunkl equation on the $2$-sphere is cast in the framework of the Racah problem for the Hopf algebra $sl_{-1}(2)$. The related Dunkl-Laplace operator is shown to correspond to a quadratic expression in the total Casimir operator of the tensor product of three irreducible $sl_{-1}(2)$-modules. The operators commuting with the Dunkl Laplacian are seen to coincide with the intermediate Casimir operators and to realize a central extension of the Bannai--Ito (BI) algebra. Functions on $S^2$ spanning irreducible modules of the BI algebra are constructed and given explicitly in terms of Jacobi polynomials. The BI polynomials occur as expansion coefficients between two such bases composed of functions separated in different coordinate systems.
\end{abstract}

\maketitle
\section{Introduction}
The purpose of this paper is to establish a relation between Dunkl harmonic analysis on the 2-sphere and the representation theory of $sl_{-1}(2)$, an algebra obtained as a $q\rightarrow -1$ limit of the quantum algebra $\mathcal{U}_{q}(\mathfrak{sl}_2)$. The Dunkl-Laplace operator on $S^2$ associated to the Abelian reflection group $\mathbb{Z}_2^{3}\cong \mathbb{Z}_2\times \mathbb{Z}_2\times \mathbb{Z}_2$ will be expressed as a quadratic polynomial in the total Casimir operator of the tensor product of three irreducible $sl_{-1}(2)$-modules. The operators commuting with the Dunkl Laplacian will be identified with the intermediate Casimir operators arising in the three-fold tensor product. On eigensubspaces of the Dunkl Laplacian, these intermediate Casimir operators will be shown to generate the Bannai--Ito algebra, which is the algebraic structure behind the Racah problem of $sl_{-1}(2)$. Functions on the 2-sphere providing bases for irreducible modules of the Bannai--Ito algebra will be constructed. It will be shown that the Bannai--Ito polynomials arise here as expansion coefficients between elements of such bases associated to the separation of variables in different spherical coordinate systems.

We first provide background on the entities involved here:  the $\mathbb{Z}_{2}^{3}$ Dunkl Laplacian and its restriction to the 2-sphere, the $sl_{-1}(2)$ algebra and its Hopf algebra structure and the Bannai--Ito algebra and the associated Bannai--Ito polynomials.
\subsection{The $\mathbb{Z}_2^3$ Dunkl-Laplacian on $S^2$}
The Dunkl operators and Laplacian were introduced by Dunkl in \cite{Dunkl-1988,Dunkl-1989-01}, where a framework for multivariate analysis based on finite reflection groups was developed. These operators have since found a vast number of applications in diverse fields including harmonic analysis and integral transforms \cite{DeBie-2011,Opdam-1998,Rosler-2003}, orthogonal polynomials and special functions \cite{Dunkl-2001}, stochastic processes \cite{Rosler-2008} and quantum integrable/superintegrable systems \cite{Genest-2013-04,CMS-2000}. In the case of the Abelian reflection group $\mathbb{Z}_2^{3}$, the Dunkl operators $\mathcal{D}_i$, $i=1,2,3$, associated to each copy of the reflection group $\mathbb{Z}_2$ are defined by
\begin{align}
\label{Dunkl-D}
\mathcal{D}_i=\partial_{x_i}+\frac{\mu_i}{x_i}(1-R_i),
\end{align}
with $\mu_i>-1/2$ a real parameter, $\partial_{x_i}$ the partial derivative with respect to the variable $x_i$ and $R_i$ the reflection operator in the $x_i=0$ plane, i.e. $R_if(x_i)=f(-x_i)$. The Dunkl Laplacian associated to the $\mathbb{Z}_2^{3}$ group is defined by
\begin{align}
\label{Dunkl-Laplacian}
\Delta=\mathcal{D}_1^2+\mathcal{D}_2^2+\mathcal{D}_3^2,
\end{align}
and has the following expression:
\begin{align*}
\Delta=\sum_{i=1}^{3}\partial_{x_i}^2+\frac{2\mu_i}{x_i}\partial_{x_i}-\frac{\mu_i}{x_i^2}(1-R_{i}).
\end{align*}
Since the reflections $R_i$, $i=1,2,3$, are special rotations in $O(3)$, the Dunkl Laplacian \eqref{Dunkl-Laplacian}, like the standard Laplace operator in three variables, separates in the usual spherical coordinates
\begin{align}
\label{Spherical-Coordinates}
x_1=r\sin\theta\cos \phi,\quad x_2=r\sin\theta\sin\phi,\quad x_3=r\cos\theta,
\end{align}
with $0\leq \theta \leq \pi$ and $0\leq \phi \leq 2\pi$. The operator $\Delta$ can thus be restricted to functions defined on the unit sphere. Let $\Delta_{S^{2}}$ denote the angular part of the Dunkl Laplacian \eqref{Dunkl-Laplacian}; one has
\begin{align}
\label{Dunkl-Laplacian-S2}
\Delta_{S^2}=L_{\theta}+\frac{1}{\sin^{2}\theta}M_{\phi},
\end{align}
where
\begin{align*}
L_{\theta}=\frac{1}{\sin\theta}\partial_{\theta}\,(\sin\theta\,\partial_{\theta})+2\left(\frac{\mu_1+\mu_2}{\tan\theta}-\mu_3\,\tan \theta\right)\partial_{\theta}-\frac{\mu_3}{\cos^2\theta}(1-R_{3}),
\end{align*}
and
\begin{align*}
M_{\phi}=\partial_{\phi}^2+2\left(\frac{\mu_2}{\tan \phi}-\mu_1\,\tan\phi\right)\partial_{\phi}-\frac{\mu_1}{\cos^2\phi}(1-R_1)-\frac{\mu_2}{\sin^2\phi}(1-R_2),
\end{align*}
as can be directly checked by expanding \eqref{Dunkl-Laplacian} in spherical coordinates.
\subsection{The Hopf algebra $\mathbf{sl_{-1}(2)}$} 
The $sl_{-1}(2)$ algebra was introduced in \cite{Tsujimoto-2011-10} as the $q\rightarrow -1$ limit of the quantum algebra $\mathcal{U}_{q}(\mathfrak{sl}_2)$ \cite{Vilenkin-1991}. It is defined as the associative algebra (over $\mathbb{C}$) with generators $A_{\pm}$, $A_0$ and $P$ satisfying the relations
\begin{align}
\label{SL}
[A_0,A_{\pm}]=\pm A_{\pm},\; [A_0,P]=0,\; \{A_{+},A_{-}\}=2A_0,\;\{A_{\pm},P\}=0,\; P^2=1,
\end{align}
where $[x,y]=xy-yx$ stands for the commutator. This algebra admits the following Casimir operator, which commutes with all generators:
\begin{align}
\label{Casimir}
C=A_{+}A_{-}P-A_{0}P+P/2.
\end{align}
The $sl_{-1}(2)$ algebra can be endowed with the structure of a Hopf algebra. One introduces the comultiplication $\Delta:sl_{-1}(2)\rightarrow sl_{-1}(2)\otimes sl_{-1}(2)$, the counit $\epsilon: sl_{-1}(2)\rightarrow \mathbb{C}$ and the coinverse (antipode) $\sigma:sl_{-1}(2)\rightarrow sl_{-1}(2)$ defined by the formulas 
\begin{gather}
\Delta(A_0)=A_0\otimes 1+1 \otimes A_0,\quad \Delta(A_{\pm})=A_{\pm}\otimes P+1\otimes A_{\pm},\quad \Delta(P)=P\otimes P,
\nonumber
\\
\label{CoProd}
\epsilon(1)=\epsilon(P)=1,\qquad \epsilon(A_{\pm})=\epsilon(A_0)=0,
\\
\sigma(1)=1,\quad \sigma(P)=P,\quad \sigma(A_0)=-A_0,\quad \sigma(A_{\pm})=PA_{\pm}.
\nonumber
\end{gather}
It is verified that the definitions \eqref{CoProd} comply with the conditions required for a Hopf algebra \cite{Underwood-2011}. It is worth pointing out that the operators $A_{\pm}$, $A_0$ also satisfy the defining relations of the parabosonic algebra for a single paraboson (see \cite{Daska-2000-02}).
\subsection{The Bannai--Ito algebra and polynomials}
The Bannai--Ito algebra was introduced in \cite{Tsujimoto-2012-03} as the algebraic structure encoding the bispectrality property of the Bannai--Ito polynomials. It is defined as the associative algebra (over $\mathbb{C}$) generated by $K_1$, $K_2$ and $K_3$ satisfying the relations
\begin{align}
\label{BI-Algebra}
\{K_1,K_2\}=K_3+\alpha_3,\quad \{K_2,K_3\}=K_1+\alpha_1,\quad \{K_3,K_1\}=K_2+\alpha_2,
\end{align}
where $\{x,y\}=xy+yx$ stands for the anticommutator and where $\alpha_i$, $i=1,2,3$, are real structure constants. In \cite{Tsujimoto-2012-03}, the algebra was introduced with the structure constants expressed as follows in terms of four real parameters $\rho_1$, $\rho_2$, $r_1$, $r_2$:
\begin{align*}
\alpha_1=4(\rho_1\rho_2+r_1r_2),\quad \alpha_2=2(\rho_1^2+\rho^2-r_1^2-r_2^2),\quad \alpha_3=4(\rho_1\rho_2-r_1r_2),
\end{align*}
and the generators had the form
\begin{align*}
K_1=2\mathcal{L}+(g+1/2),\quad K_2=y,
\end{align*}
with $g=\rho_1+\rho_2-r_1-r_2$ and $\mathcal{L}$ the difference operator
\begin{align*}
\mathcal{L}=\frac{(y-\rho_1)(y-\rho_2)}{2y}(1-R_{y})+\frac{(y-r_1+1/2)(y-r_2+1/2)}{2y+1}(T^{+}_yR_y-1),
\end{align*}
where $R_y f(y)=f(-y)$, $T^{+}_{y}f(y)=f(y+1)$. The operator $\mathcal{L}$ is the most general self-adjoint first order difference operator with reflections that stabilizes the space of polynomials of a given degree. As shown in \cite{Tsujimoto-2012-03}, the operator $\mathcal{L}$ admits as eigenfunctions the Bannai--Ito polynomials $B_{n}(y)$, which were introduced in a combinatorial context by Bannai and Ito in \cite{Bannai-1984}. Their three-term recurrence relation was derived in \cite{Tsujimoto-2012-03} using the BI algebra \eqref{BI-Algebra} and reads
\begin{align}
\label{BI-Recurrence}
x B_{n}(y)= B_{n+1}(y)+(\rho_1-A_{n}-C_{n})B_{n}(y)+A_{n-1}C_{n}B_{n-1}(y),
\end{align}
where the initial conditions $B_{-1}(x)=0$, $B_{0}(x)=1$ hold and where the recurrence coefficients $A_n$, $C_n$ are given by
\begin{subequations}
\label{Recu-Coeff}
\begin{align}
A_{n}&=
\begin{cases}
\frac{(n+2\rho_1-2r_1+1)(n+2\rho_1-2r_2+1)}{4(n+\rho_1+\rho_2-r_1-r_2+1)}, & \text{$n$ is even},
\\
\frac{(n+2\rho_1+2\rho_2-2r_1-2r_2+1)(n+2\rho_1+2\rho_2+1)}{4(n+\rho_1+\rho_2-r_1-r_2+1)}, & \text{$n$ is odd},
\end{cases}
\\
C_{n}&=
\begin{cases}
-\frac{n(n-2r_1-2r_2)}{4(n+\rho_1+\rho_2-r_1-r_2)}, & \text{$n$ is even},
\\
-\frac{(n+2\rho_2-2r_2)(n+2\rho_2-2r_1)}{4(n+\rho_1+\rho_2-r_1-r_2)}, & \text{$n$ is odd}.
\end{cases}
\end{align}
\end{subequations}
The polynomials $B_n(y)$ defined by \eqref{BI-Recurrence} are $q\rightarrow -1$ limits of either the Askey-Wilson \cite{Tsujimoto-2012-03} or the $q$-Racah polynomials \cite{Bannai-1984}. They obey a discrete and finite orthogonality relation of the form
\begin{align*}
\sum_{s=0}^{N}w_{s}B_{n}(y_s)B_{m}(y_s)=h_{n}\delta_{nm},
\end{align*}
where the expressions for the grid points $y_s$, the measure $w_{s}$ and the normalization constant $h_n$ depend on a set of relations between the parameters. For the complete picture, one may consult the references \cite{Genest-2013-02-1,Tsujimoto-2012-03}.
\subsection{Outline}
Here is an outline of the paper.
\begin{itemize}
\item Section II: Irreducible $sl_{-1}(2)$-modules (positive-discrete series), Realization with Dunkl operators, Racah problem, Intermediate Casimir operators, Relation between the total Casimir and $\Delta_{S^2}$, Spectra of the total and intermediate Casimir operators
\item Section III: Commutant of $\Delta_{S^2}$, Bannai--Ito algebra, Finite-dimensional irreducible representations of the BI algebra
\item Section IV: Dunkl spherical harmonics for $\mathbb{Z}_2^{3}$, $S^2$ basis functions for irreducible modules of the BI algebra, BI polynomials as expansion coefficients between basis functions
\end{itemize}
\section{Racah problem of $sl_{-1}(2)$ and $\Delta_{S^2}$}
In this section, irreducible $sl_{-1}(2)$-modules of the positive-discrete series and their realizations in terms of the Dunkl operators \eqref{Dunkl-D} are given. The Racah problem is presented and the intermediate and total Casimir operators are defined. The main result on the relation between the total Casimir operator and the Dunkl Laplacian on $S^2$ is presented. Moreover, the spectrum of the Dunkl Laplacian is recovered algebraically using this relation.
\subsection{Representations of the positive-discrete series and their realization in terms of Dunkl operators}
Let $\epsilon$ and $\nu$ be real parameters such that $\epsilon^2=1$ and $\nu>-1/2$ and denote by $V^{(\epsilon,\nu)}$ the infinite-dimensional vector space spanned by the orthonormal basis vectors $e_{n}^{(\epsilon,\nu)}$ with $n$ a non-negative integer. An irreducible $sl_{-1}(2)$-module of the positive-discrete series is obtained by endowing $V^{(\epsilon,\nu)}$ with the actions \cite{Tsujimoto-2011-10}:
\begin{subequations}
\label{Representations}
\begin{gather}
A_0\,e_{n}^{(\epsilon,\nu)}=(n+\nu+1/2)\,e_{n}^{(\epsilon,\nu)},\quad P\,e_{n}^{(\epsilon,\nu)}=\epsilon (-1)^{n}\,e_{n}^{(\epsilon,\nu)},
\\
A_{+}\,e_{n}^{(\epsilon,\nu)}=\sqrt{[n+1]_{\nu}}\,e_{n+1}^{(\epsilon,\nu)},\quad A_{-}\,e_{n}^{(\epsilon,\nu)}=\sqrt{[n]_{\nu}}\,e_{n-1}^{(\epsilon,\nu)},
\end{gather}
\end{subequations}
where $[n]_{\nu}$ is defined by
\begin{align*}
[n]_{\nu}=n+\nu(1-(-1)^{n}).
\end{align*}
It is directly seen that for $\nu>-1/2$, $V^{(\epsilon,\mu)}$ is an irreducible module. Furthermore, it is observed that on this module the spectrum of $A_0$ is strictly positive and the operators $A_{\pm}$ are adjoint one of the other. As expected from Schur's lemma, the Casimir operator \eqref{Casimir} of $sl_{-1}(2)$ acts a multiple of the identity on $V^{(\epsilon,\nu)}$:
\begin{align}
\label{Ultra}
C\,e_{n}^{(\epsilon,\nu)}=-\epsilon\,\nu \,e_{n}^{(\epsilon,\nu)}.
\end{align}
The $sl_{-1}(2)$-module $V^{(\epsilon,\nu)}$ can be realized using Dunkl operators. Indeed, for each variable $x_i$, $i=1,2,3$, one can check that the operators
\begin{align}
\label{Dunkl-Realization}
A_0^{(i)}=-\frac{1}{2}\mathcal{D}_i^2+\frac{1}{2}x_i^2,\quad A_{\pm}^{(i)}=\frac{1}{\sqrt{2}}(x_i\mp \mathcal{D}_i), \quad P^{(i)}=R_i,
\end{align}
where $\mathcal{D}_i$ and $R_i$ are as in \eqref{Dunkl-D}, satisfy the defining relations \eqref{SL} of $sl_{-1}(2)$. The Casimir operator $C^{(i)}$ becomes
\begin{align}
\label{Initial}
C^{(i)}=A_{+}^{(i)}A_{-}^{(i)}P^{(i)}-A_0^{(i)}P^{(i)}+P^{(i)}/2=-\mu_i,
\end{align}
and hence the operators \eqref{Dunkl-Realization} for $i=1,2,3$ realize the irreducible module $V^{(\epsilon,\nu)}$ with $\epsilon=\epsilon_i=1$ and $\nu=\mu_i$. The orthonormal basis vectors $e_{n}^{(\epsilon_i,\nu_i)}(x_i)$ in this realization are expressed in terms of the generalized Hermite polynomials (see for example \cite{Genest-2013-04,Rosenblum-1994}) and the space $V^{(\epsilon_i,\nu_i)}$ with $\epsilon_i=1$ and $\nu_i=\mu_i$ is the $L^2$ space of square integrable functions of argument $x_i$ with respect to the orthogonality measure of the generalized Hermite polynomials \cite{Rosenblum-1994}; we shall denote it by $L^{2}_{\mu_i}$.
\subsection{The Racah problem, Casimir operators and $\Delta_{S^2}$}
The Racah problem for $sl_{-1}(2)$-modules of the positive-discrete series arises when the decomposition in irreducible components of the module $V=V^{(\epsilon_1,\nu_1)}\otimes V^{(\epsilon_2,\nu_2)}\otimes V^{(\epsilon_3,\nu_3)}$ is considered. The action of the $sl_{-1}(2)$ generators on $V$ is prescribed by the coproduct structure \eqref{CoProd} and one has for $v\in V$
\begin{gather}
\label{Tensor}
A_0v=(1\otimes \Delta)\Delta(A_0)v,\quad Pv=(1\otimes \Delta)\Delta(P)v,\quad A_{\pm}v=(1\otimes \Delta)\Delta(A_{\pm})v.
\end{gather}
Note that $(1\otimes \Delta)\Delta=(\Delta\otimes 1)\Delta$ since $\Delta$ is coassociative. In the realization \eqref{Dunkl-Realization}, the module $V$ (with $\epsilon_i=1$ and $\nu_i=\mu_i$) involves functions of the three independent variables $x_1$, $x_2$, $x_3$. The operators satisfying the $sl_{-1}(2)$ relations and acting on functions $f(x_1,x_2,x_3)$ in $L^2_{\mu_1}\otimes L^2_{\mu_2}\otimes L^{2}_{\mu_3}$ are obtained from \eqref{Dunkl-Realization} and \eqref{Tensor}:
\begin{gather}
\label{Full-Algebra}
\begin{aligned}
\widetilde{A}_0=A_0^{(1)}+A_0^{(2)}+A_0^{(3)},\quad \widetilde{P}=P^{(1)}P^{(2)}P^{(3)},
\\
\widetilde{A}_{\pm}=A_{\pm}^{(1)}P^{(2)}P^{(3)}+A_{\pm}^{(2)}P^{(3)}+A_{\pm}^{(3)}.
\end{aligned}
\end{gather}
In combining the modules $V^{(\epsilon_i,\nu_i)}$, $i=1,2,3$, three types of Casimir operators can be distinguished. The three 
\emph{initial} Casimir operators are those attached to each components $V^{(\epsilon_i,\mu_i)}$ of $V$ and act as multiplication by $-\epsilon_i\nu_i$ as per \eqref{Ultra}. In the realization \eqref{Dunkl-Realization}, these are the $C^{(i)}$ given in \eqref{Initial}. The two \emph{intermediate} Casimir operators are associated to the two equivalent factorizations
\begin{align}
\label{Decomposition}
V=(V^{(\epsilon_1,\nu_1)}\otimes V^{(\epsilon_2,\nu_2)})\otimes V^{(\epsilon_3,\nu_3)}=V^{(\epsilon_1,\nu_1)}\otimes(V^{(\epsilon_2,\nu_2)}\otimes V^{(\epsilon_3,\nu_3)}),
\end{align}
and correspond to the operators
\begin{align}
\label{Inter-Abs}
\Delta(C)\otimes 1\qquad \text{and}\qquad 1\otimes \Delta(C),
\end{align}
where $\Delta(C)$ is obtained from \eqref{Casimir} and \eqref{CoProd}. In the realization \eqref{Dunkl-Realization}, these shall be denoted $C^{(ij)}$ with $(ij)=(12),(23)$ and are given by
\begin{align}
\label{Intermediate}
\begin{aligned}
C^{(ij)}&=(A_{+}^{(i)}P^{(i)}+A_{+}^{(j)}P^{(j)})(A_{-}^{(i)}P^{(i)}+A_{-}^{(j)}P^{(j)})
\\
&\quad -(A_0^{(i)}+A_0^{(j)})P^{(i)}P^{(j)}+P^{(i)}P^{(j)}.
\end{aligned}
\end{align}
The \emph{total} Casimir operator is connected to the whole module $V$ and is of the form $(1\otimes \Delta)\Delta(C)$. In the realization \eqref{Dunkl-Realization}, the total Casimir is denoted $\widetilde{C}$ and reads
\begin{align}
\label{Total}
\widetilde{C}=\widetilde{A}_{+}\widetilde{A}_{-}\widetilde{P}-\widetilde{A}_0\widetilde{P}+\widetilde{P}/2.
\end{align}
with $\widetilde{A}_0$, $\widetilde{A}_{\pm}$ and $\widetilde{P}$ given by \eqref{Full-Algebra}. Note that $\widetilde{C}$ does not act as a multiple of the identity on $V$ since in general $V$ is not irreducible.
\begin{remark}
By construction, the total Casimir operator $\widetilde{C}$ commutes with both the initial and intermediate Casimir operators. Moreover, it is obvious that the two intermediate Casimir operators commute with the initial Casimir operators, but do not commute amongst themselves.
\end{remark}
We now relate the total Casimir operator $\widetilde{C}$ to the Dunkl Laplacian operator $\Delta_{S^2}$ on the 2-sphere.
\begin{proposition}
Let $\Omega$ be the following element:
\begin{align}
\label{Def-Omega}
\Omega=\widetilde{C}\widetilde{P},
\end{align}
where $\widetilde{C}$ and $\widetilde{P}$ are respectively given by \eqref{Full-Algebra} and \eqref{Total} in the realization \eqref{Dunkl-Realization}. One has
\begin{align}
\label{Main-Result}
-\Delta_{S^2}=\Omega^2+\Omega-(\mu_1+\mu_2+\mu_3)(\mu_1+\mu_2+\mu_3+1).
\end{align}
\end{proposition}
\begin{proof}
The relation is obtained by expanding the total Casimir operator \eqref{Total} using \eqref{Dunkl-Realization} and by writing the resulting operator in the coordinates \eqref{Spherical-Coordinates}.
\end{proof}
The fact that $\Omega$ is a purely angular operator can be understood algebraically as follows. Consider the element $\widetilde{X}$ defined by
\begin{align*}
\widetilde{X}=\frac{1}{\sqrt{2}}\left(\widetilde{A}_{+}+\widetilde{A}_{-}\right).
\end{align*}
It is directly checked that $\widetilde{X}$ anticommutes with $\Omega$, that is $\{\Omega,\widetilde{X}\}=0$. It thus follows that $\widetilde{X}^2$ commutes with $\Omega$. Using the expressions \eqref{Full-Algebra} for the operators $\widetilde{A}_{\pm}$ in the realization \eqref{Dunkl-Realization}, it is easily seen that
\begin{align*}
\widetilde{X}^2=x_1^2+x_2^2+x_3^2.
\end{align*}
Hence $\Omega$ commutes with the ``radius'' operator, which means that it can only be an angular operator.
\subsection{Spectrum of $\Delta_{S^2}$ from the Racah problem}
The relation \eqref{Main-Result} can be exploited to algebraically derive the spectrum of $\Delta_{S^2}$ from that of $\Omega$ using the eigenvalues of the intermediate Casimir operators. In view of \eqref{Inter-Abs}, these eigenvalues can be found from those of $\Delta(C)$ on $V^{(\epsilon_i,\nu_i)}\otimes V^{(\epsilon_j,\nu_j)}$ (see also \cite{Genest-2012,Genest-2013-02,Tsujimoto-2011-10} where this problem was considered). Upon examining the action of $\Delta(A_0)$ on the direct product basis, one obtains using \eqref{Representations} the following direct sum decomposition of $V^{(\epsilon_i,\nu_i)}\otimes V^{(\epsilon_j,\nu_j)}$ has a vector space:
\begin{align*}
V^{(\epsilon_i,\nu_i)}\otimes V^{(\epsilon_j,\nu_j)}=\bigoplus_{n=0}^{\infty}U_{n},
\end{align*}
where $U_{n}$ are the $(n+1)$-dimensional eigenspaces of $\Delta(A_0)$ with eigenvalue $n+\nu_i+\nu_j+1$. Since $\Delta(C)$ commutes with $\Delta(A_0)$, the action of $\Delta(C)$ stabilizes $U_{n}$.
\begin{lemma}
The eigenvalues $\lambda_{I}$ of $\Delta(C)$ on $U_{n}$ are given by
\begin{align*}
\lambda_{I}(k)=(-1)^{k+1}\epsilon_i\epsilon_j(k+\nu_i+\nu_j+1/2),\qquad k=0,\ldots n.
\end{align*}
\end{lemma}
\begin{proof}
By induction on $n$. The $n=0$ case is verified by acting with $\Delta(C)$ on the single basis vector $e_{0}^{(\epsilon_i,\nu_i)}\otimes e_0^{(\epsilon_j,\nu_j)}$ of $U_0$. Suppose that the result holds at level $n-1$. Using the fact that $\Delta(C)$ and $\Delta(A_{+})$ commute and the induction hypothesis, one obtains from the action of $\Delta(A_{+})$ on $U_{n-1}$ eigenvectors of $\Delta(C)$ in $U_{n}$ with eigenvalues $\lambda_I(k)$ for $k=0,\ldots,n-1$. Let $v\in U_{n}$ be such that $\Delta(A_-)v=0$. Such a vector can explicitly be constructed in the direct product basis by solving the corresponding two-term recurrence relation. It is verified that $v$ is an eigenvector of $\Delta(P)$ with eigenvalue $(-1)^{n}\epsilon_i\epsilon_j$ and of $\Delta(C)$ with eigenvalue $\lambda_I(n)$.
\end{proof}
As a direct corollary one has the following decomposition of the tensor product module in irreducible components:
\begin{align}
\label{Sub-Decomposition}
V^{(\epsilon_i,\nu_i)}\otimes V^{(\epsilon_j,\nu_j)}=\bigoplus_{k}V^{(\epsilon_{ij}(k),\nu_{ij}(k))},
\end{align}
with 
\begin{align}
\label{Decompo}
\epsilon_{ij}(k)=(-1)^{k}\epsilon_i\epsilon_j,\quad \nu_{ij}(k)=k+\nu_i+\nu_j+1/2,\qquad k\in \mathbb{N}.
\end{align}
The eigenvalues of the total Casimir operator $(1\otimes \Delta)\Delta(C)$ on $V$ are obtained by using twice the decomposition \eqref{Sub-Decomposition} and Lemma 1 on \eqref{Decomposition}. It is readily seen performing these decompositions on the LHS of \eqref{Decomposition} that the eigenvalues $\lambda_{T}$ of  the total Casimir operator are given by
\begin{align}
\label{Total-Eigen-1}
\lambda_{T}=(-1)^{k+1}\epsilon_{12}(\ell)\epsilon_{3}(k+\nu_{12}(\ell)+\nu_3+1/2),\qquad k,\ell \in \mathbb{N}.
\end{align}
A similar formula involving $\epsilon_{23}$ and $\nu_{23}$ is obtained by considering instead the RHS of \eqref{Decomposition}. Upon using \eqref{Decompo}, the eigenvalues $\lambda_{T}$ can be cast in the form
\begin{align}
\label{Total-Eigen-2}
\lambda_{T}(N)=-\epsilon(N)\nu(N),
\end{align}
with $N$ a non-negative integer and
\begin{align}
\label{Total-Eigen-3}
\epsilon(N)=(-1)^{N}\epsilon_1\epsilon_2\epsilon_3,\qquad \nu(N)=(N+\nu_1+\nu_2+\nu_3+1).
\end{align}
The formula \eqref{Total-Eigen-2} and \eqref{Total-Eigen-3} indicate which irreducible modules appear in the decomposition of $V$. The multiplicity of $V^{(\epsilon(N),\nu(N))}$ in this decomposition is $N+1$ since for a given value of $N$ there are $N+1$ possible eigenvalues of the intermediate Casimir operators; the decomposition formula for $V$ is thus
\begin{align}
\label{Full-Decompo}
V=\bigoplus_{N=0}^{\infty}m_{N}V^{(\epsilon(N),\nu(N))},
\end{align}
where $m_{N}=N+1$ and where $\epsilon(N)$, $\nu(N)$ are given by \eqref{Total-Eigen-3}.

Returning to the realization \eqref{Full-Algebra} of the module $V$ with $\epsilon_i=1$ and $\nu_i=\mu_i$, the eigenvalues of $\Omega=\widetilde{C}\widetilde{P}$ are readily obtained. Recalling \eqref{Ultra}, it follows from \eqref{Total-Eigen-2} and \eqref{Total-Eigen-3} that the eigenvalues $\omega_{N}$ of $\Omega$ are
\begin{align}
\label{Total-Eigen-4}
\omega_{N}=-(N+\mu_1+\mu_2+\mu_3+1),
\end{align}
where $N$ is a non-negative integer. The relation \eqref{Main-Result} then leads to the following.
\begin{proposition}
The eigenvalues $\delta$ of the Dunkl Laplacian $\Delta_{S^2}$ on the 2-sphere are indexed by the non-negative integer $N$ and have the expression
\begin{align}
\label{Eigen-Dunkl}
\delta_{N}=-N(N+2\mu_1+2\mu_2+2\mu_3+1).
\end{align}
\end{proposition}
\begin{proof}
By proposition 1 and the above considerations.
\end{proof}
The eigenvalues of proposition $2$ are in accordance with those obtained in \cite{Dunkl-2001}. It is seen that upon specializing \eqref{Eigen-Dunkl} to $\mu_1=\mu_2=\mu_3=0$, one recovers the spectrum of the standard Laplacian on the 2-sphere. It is worth mentioning that the formula \eqref{Eigen-Dunkl} does not provide information on the degeneracy of the eigenvalues. This question will be discussed in the following.
\section{Commutant of $\Delta_{S^2}$ and the Bannai--Ito algebra}
In this section, the operators commuting with the Dunkl Laplacian on the 2-sphere are exhibited and are shown to generate a central extension of the Bannai--Ito algebra. The eigensubspaces corresponding to the simultaneous diagonalization of $\Delta_{S^2}$ and $\Omega$ are seen to support finite-dimensional irreducible representations of the BI algebra and the matrix elements of these representations are constructed.
\subsection{Commutant of $\Delta_{S^2}$ and symmetry algebra}
The operators that commute with the Dunkl Laplacian $\Delta_{S^2}$ on the 2-sphere, referred to as the \emph{symmetries} of $\Delta_{S^2}$, can be obtained from the relation \eqref{Main-Result} and the framework provided by the Racah problem of $sl_{-1}(2)$. By construction, the intermediate Casimir operators \eqref{Intermediate} commute with the total Casimir \eqref{Total} and with the involution $\widetilde{P}$. As a consequence of \eqref{Main-Result}, one thus has
\begin{align*}
[\Delta_{S^2}, C^{(12)}]=[\Delta_{S^2},C^{(23)}]=0.
\end{align*}
Let $K_1$, $K_3$ be the following operators:
\begin{align}
\label{K-Def}
K_1=-C^{(23)},\qquad K_3=-C^{(12)},
\end{align}
which obviously commute with the Dunkl Laplacian on $S^2$. Upon using \eqref{Dunkl-Realization} and \eqref{Intermediate}, the symmetries $K_1$, $K_3$ are seen to have the expressions
\begin{subequations}
\label{BI-Generators}
\begin{align}
K_1&=(x_2\mathcal{D}_3-x_3 \mathcal{D}_2)R_2+\mu_2 R_3+\mu_3 R_2+(1/2) R_2R_3,
\\
K_3&=(x_1\mathcal{D}_2-x_2 \mathcal{D}_1)R_{1}+\mu_1 R_2+\mu_2 R_1+(1/2) R_1R_2,
\end{align}
where $\mathcal{D}_i$ and $R_i$ are given by \eqref{Dunkl-D}. Consider the operator $K_2$ defined by
\begin{align}
\label{C}
K_2&=(x_1\mathcal{D}_3-x_3 \mathcal{D}_1)R_{1}R_{2}+\mu_1 R_3+\mu_3 R_1+(1/2)R_1R_3.
\end{align}
\end{subequations}
It is verified by an explicit calculation that $K_2$ is also a symmetry of the Dunkl-Laplacian $\Delta_{S^2}$, i.e. $[\Delta_{S^2},K_2]=0$. 
\begin{remark}
Note that $K_2$ does not correspond to an intermediate Casimir operator since it has a non-trivial action on all three variables $x_1,x_2,x_3$.
\end{remark}
The three operators $K_i$, $i=1,2,3$, and the operator $\Omega$ given by \eqref{Def-Omega} are not independent from one another. As a matter of fact, one has
\begin{align*}
\Omega=-K_1 R_2R_3-K_2R_1 R_3-K_3 R_1R_2+\mu_1 R_1+\mu_2R_2+\mu_3 R_3+1/2.
\end{align*}
We now give the \emph{symmetry algebra} generated by the operators commuting with the Dunkl-Laplace operator $\Delta_{S^2}$ on the 2-sphere.
\begin{proposition}
Let $\Delta_{S^2}$ be the Dunkl Laplacian \eqref{Dunkl-Laplacian-S2} on the 2-sphere  and let $\widetilde{C}$ and $K_i$, $i=1,2,3$  be given by \eqref{Total} and \eqref{BI-Generators}, respectively. One has
\begin{align*}
[\Delta_{S^2}, K_i]=[\Delta_{S^2}, \widetilde{C}]=0.
\end{align*}
and the symmetry algebra of $\Delta_{S^2}$ is
\begin{subequations}
\label{Sym-Algebra}
\begin{align}
\{K_1,K_2\}&=K_3-2\mu_3\widetilde{C}+2\mu_1\mu_2,
\\
\{K_2,K_3\}&=K_1-2\mu_1\widetilde{C}+2\mu_2\mu_3,
\\
\{K_3,K_1\}&=K_2-2\mu_2 \widetilde{C}+2\mu_1\mu_3.
\end{align}
\end{subequations}
\end{proposition}
\begin{proof}
By an explicit calculation using \eqref{Dunkl-Laplacian-S2} and \eqref{BI-Generators}.
\end{proof}
The algebra \eqref{Sym-Algebra} corresponds to a central extension of the Bannai--Ito algebra \eqref{BI-Algebra} by the total Casimir operator $\widetilde{C}$. Since $\widetilde{C}$ (and $\Omega$) commutes with $\Delta_{S^2}$, there is a basis in which they are both diagonal. From \eqref{Total-Eigen-3} and \eqref{Total-Eigen-4},  it follows that the eigenvalues of $\widetilde{C}$ are of the form $-\epsilon\, \mu$ with
\begin{align}
\label{Choupette}
\epsilon=(-1)^{N},\qquad \mu=(N+\mu_1+\mu_2+\mu_3+1).
\end{align}
For a given $N$, the $\Delta_{S^2}$-eigenspaces arising under the joint diagonalization of $\Delta_{S^2}$ and $\widetilde{C}$ (or $\Omega$) are $(N+1)$-dimensional as per the decomposition \eqref{Full-Decompo} of the tensor product module $V$ in irreducible components. Hence the eigenvalues $\delta_{N}$ of $\Delta_{S^2}$ given by \eqref{Eigen-Dunkl} are at least $(N+1)$-fold degenerate. It can be seen that this degeneracy is in fact higher. Indeed, $\Delta_{S^2}$ commutes with every reflection operator $R_i$, but $\widetilde{C}$ (and $\Omega$) only commute with their product $R_1R_2R_3$. Consequently one can obtain eigenfunctions of $\Delta_{S^2}$ with eigenvalue $\delta_{N}$ that are not eigenfunctions of $\widetilde{C}$ by applying any reflection $R_i$ on a given eigenfunction of $\widetilde{C}$. It is known \cite{Dunkl-2001} that the eigenspaces corresponding to the eigenvalue $\delta_{N}$ are in fact $(2N+1)$-fold degenerate, as shall be seen in Section 4.

Notwithstanding the degeneracy question, it follows from Proposition 3 and \eqref{Choupette} that the eigensubspaces of the Laplace-Dunkl operator corresponding to the simultaneous diagonalization of $\Delta_{S^2}$ and $\widetilde{C}$ support an $(N+1)$-dimensional module of the Bannai--Ito algebra \eqref{BI-Algebra} with structure constants taking the values
\begin{align}
\label{rea}
\alpha_1=2(\mu_1\mu+\mu_2\mu_3),\quad \alpha_2=2(\mu_1\mu_3+\mu_2\mu),\quad \alpha_3=2(\mu_1\mu_2+\mu_3\mu),
\end{align}
where $\mu=(-1)^{N}(N+\mu_1+\mu_2+\mu_3+1)$. The Casimir operator $\mathbf{K}^2=K_1^2+K_2^2+K_3^2$ of the Bannai--Ito algebra can be expressed in terms of $\widetilde{C}$ as follows:
\begin{align*}
\mathbf{K}^2=\widetilde{C}^2+\mu_1^2+\mu_2^2+\mu_3^2-1/4,
\end{align*}
and hence using \eqref{Choupette} one has
\begin{align}
\label{rea-2}
\mathbf{K}^2=\mu_1^2+\mu_2^2+\mu_3^2+\mu^2-1/4.
\end{align}
The realization \eqref{rea}, \eqref{rea-2} of the Bannai--Ito algebra corresponds to the one arising in the Racah problem for $sl_{-1}(2)$ studied in \cite{Genest-2012}. We shall now obtain the matrix elements of the generators in this realization.
\subsection{Irreducible modules of the Bannai--Ito algebra}
We begin by examining the representations of \eqref{BI-Algebra} with structure constants \eqref{rea} in the eigenbasis $\{\psi_{k}\}_{k=0}^{N}$ of $K_3$. Using the result of Lemma 1 and \eqref{K-Def}, it follows that
\begin{align}
\label{Dompe}
K_3\psi_{k}=\omega_{k}\psi_{k},\qquad \omega_{k}=(-1)^{k}(k+\mu_1+\mu_2+1/2),
\end{align}
We define the action of $K_1$ by
\begin{equation}
K_1\psi_k=\sum_{s}Z_{s,k}\psi_{s}.
\end{equation}
From the second relation of \eqref{BI-Algebra} one finds
\begin{align*}
\sum_{s}Z_{s,k}\left[(\omega_{k}+\omega_{s})^2-1\right]\psi_{s}=\left[\alpha_1+2\omega_{k}\alpha_2\right]\psi_{k}.
\end{align*}
When $s=k$, one immediately obtains
\begin{align}
\label{Dompe-2}
Z_{k,k}\equiv V_{k}=\frac{\alpha_1+2\omega_{k}\alpha_2}{4\omega_k^2-1}.
\end{align}
When $s\neq k$, one of the following conditions must hold
\begin{align*}
(\omega_{k}+\omega_{s})^2-1=0,\quad \text{or}\quad Z_{s,k}=0.
\end{align*}
In view of the formula \eqref{Dompe} for the eigenvalues $\omega_{k}$, it is directly seen that only $Z_{k+1,k}$, $Z_{k,k}$ and $Z_{k-1,k}$ can be non-vanishing. Thus one can take
\begin{align}
\label{Action-1}
K_1\psi_{k}=U_{k+1}\psi_{k+1}+V_{k}\psi_{k}+U_{k}\psi_{k-1},
\end{align}
where $V_k$ is given by \eqref{Dompe-2} and where $U_{k}$ remains to be determined. It follows from \eqref{BI-Algebra} and \eqref{Action-1} that $K_2$ has the action
\begin{align}
\label{Action-2}
K_2\psi_{k}=(-1)^{k+1}U_{k+1}+W_{k}\psi_{k}+(-1)^{k}U_{k}\psi_{k-1},
\end{align}
where $W_{k}=2\omega_k V_k-\alpha_2$. Upon using the actions \eqref{Action-1}, \eqref{Action-2} in the first relation of \eqref{BI-Algebra} and comparing the terms in $\psi_{k}$, one obtains the recurrence relation for $U_{k}^2$
\begin{align}
\label{A}
2\left\{(-1)^{k+1}U_{k+1}^2+W_{k}V_{k}+(-1)^{k}U_{k}^2\right\}=\omega_{k}+\alpha_3.
\end{align}
Acting on $\psi_{k}$ with \eqref{rea-2} and using the actions \eqref{Action-1}, \eqref{Action-2}, one finds
\begin{align}
\label{B}
\left\{\omega_{k}^2+W_{k}^2+V_{k}^2+2U_{k}^2+2 U_{k+1}^2\right\}=\mu_1^2+\mu_2^2+\mu_3^2+\mu^2-1/4.
\end{align}
The equations \eqref{A}, \eqref{B} can be used to solve for $U_{k}^2$ by eliminating $U_{k+1}^2$. Straightforward calculations then lead to the following result.
\begin{proposition}
Let $\mathcal{W}$ be the $(N+1)$-dimensional vector space spanned by the basis vectors $\psi_{k}$, $k=0,\ldots,N$, and let 
\begin{align}
\label{Mu}
\mu=(-1)^{N}(N+1+\mu_1+\mu_2+\mu_3).
\end{align}
An irreducible module for the Bannai--Ito algebra \eqref{BI-Algebra} with structure constants \eqref{rea} is obtained by endowing $\mathcal{W}$ with the actions
\begin{subequations}
\label{Actions-4}
\begin{align}
K_3\psi_{k}&=\omega_k \psi_k,
\\
K_2\psi_{k}&=(-1)^{k+1}U_{k+1}\psi_{k+1}+(2\omega_kV_k-\alpha_2)\psi_{k}+(-1)^{k}U_{k}\psi_{k-1},
\\
K_1\psi_{k}&=U_{k+1}\psi_{k+1}+V_{k}\psi_{k}+U_{k}\psi_{k-1},
\end{align}
\end{subequations}
where $\omega_{k}=(-1)^{k}(k+\mu_1+\mu_2+1/2)$, $V_{k}=\mu_2+\mu_3+1/2-B_{k}-D_{k}$ and  where $U_{k}=\sqrt{B_{k-1}D_{k}}$ with
\begin{align*}
B_{k}&=
\begin{cases}
\frac{(k+2\mu_2+1)(k+\mu_1+\mu_2+\mu_3-\mu+1)}{2(k+\mu_1+\mu_2+1)}, & \text{$k$ is even},
\\
\frac{(k+2\mu_1+2\mu_2+1)(k+\mu_1+\mu_2+\mu_3+\mu+1)}{2(k+\mu_1+\mu_2+1)}, & \text{$k$ is odd},
\end{cases}
\\
D_{k}&=
\begin{cases}
\frac{-k(k+\mu_1+\mu_2-\mu_3-\mu)}{2(k+\mu_1+\mu_2)}, & \text{$k$ is even},
\\
\frac{-(k+2\mu_1)(k+\mu_1+\mu_2-\mu_3+\mu)}{2(k+\mu_1+\mu_2)}, & \text{$k$ is odd}.
\end{cases}
\end{align*}
\end{proposition}
\begin{proof}
One verifies directly that with \eqref{Actions-4} the defining relations \eqref{BI-Algebra}, \eqref{rea} are satisfied. The irreducibility follows from the fact that $U_{k}\neq 0$ for $\mu_i>-1/2$.
\end{proof}
In view of Proposition 4, it is natural to wonder what the representation matrix elements look like in other bases, say the eigenbases of either $K_1$ or $K_2$. These elements are easily obtained from the $\mathbb{Z}_3$ symmetry of the realization \eqref{rea}, \eqref{rea-2}. Indeed, it is verified that the algebra \eqref{BI-Algebra} with \eqref{rea}, \eqref{rea-2} is left invariant by any cyclic transformation of both $\{K_1,K_2,K_3\}$ and $\{\mu_1,\mu_2,\mu_3\}$. As a consequence, the representation matrix elements in the $K_1$ or $K_2$ eigenbasis can be obtained directly from Proposition 4 by applying the permutation $\pi=(123)$ or $\pi=(123)^2$ on the generators $K_i$ and the parameters $\mu_i$.
\section{$S^2$ basis functions for irreducible Bannai--Ito modules}
In this section, a family of orthonormal functions on $S^2$ that realize bases for the Bannai--Ito modules of Proposition 4 are constructed. It is shown that the Bannai--Ito polynomials arise as the overlap coefficients between two such bases separated in different spherical coordinates.
\subsection{Harmonics for $\Delta_{S^2}$}
It is useful to give here the Dunkl spherical harmonics $Y_{N}(\theta,\phi)$ which are the regular solutions to the eigenvalue equation
\begin{align}
\label{Above}
\Delta_{S^2}Y_{N}(\theta,\phi)=\delta_{N}Y_{N}(\theta,\phi),\qquad \delta_{N}=-N(N+2\mu_1+2\mu_2+2\mu_3+1),
\end{align}
where $\Delta_{S^2}$ is given by \eqref{Dunkl-Laplacian-S2}. The solutions to \eqref{Above} are well known and are given explicitly in \cite{Dunkl-2001} in terms of the generalized Gegenbauer polynomials. We give their expressions here in terms of Jacobi polynomials. In spherical coordinates \eqref{Spherical-Coordinates}, the solutions to \eqref{Above} read
\begin{multline}
\label{Dunkl-Harmonics}
Y_{n;N}^{(e_1,e_2,e_3)}(\theta,\phi)=\eta_{n;N}^{(e_1,e_2,e_3)}\,\cos^{e_3}\theta \sin^{n}\theta\,\cos^{e_1}\phi\sin^{e_2}\phi
\\
\times P_{(N-n-e_3)/2}^{(n+\mu_1+\mu_2,\mu_3+e_3-1/2)}(\cos 2\theta)\;P_{(n-e_1-e_2)/2}^{(\mu_2+e_2-1/2,\mu_1+e_1-1/2)}(\cos 2\phi),
\end{multline}
where $e_i\in\{0,1\}$, $n$ is a non-negative integer, $\eta_{N,n}^{(e_1,e_2,e_3)}$ is a normalization factor and $P_{n}^{(\alpha,\beta)}(x)$ are the standard Jacobi polynomials \cite{Koekoek-2010}. The harmonics \eqref{Dunkl-Harmonics} satisfy
\begin{align*}
R_i\,Y_{n;N}^{(e_1,e_2,e_3)}(\theta,\phi)=(1-2e_i)Y_{n;N}^{(e_1,e_2,e_3)}(\theta,\phi).
\end{align*}
In \eqref{Dunkl-Harmonics}, it is understood that half-integer (or negative) indices in $P_{n}^{(\alpha,\beta)}(x)$ do not provide admissible solutions. Recording the admissible values of $n$ and $e_i$ for a given $N$, one finds that there are $2N+1$ solutions and
\begin{align*}
R_1 R_2 R_3 Y_{n;N}^{(e_1,e_2,e_3)}(\theta,\phi)=(-1)^{N} Y_{n;N}^{(e_1,e_2,e_3)}(\theta,\phi).
\end{align*}
The normalization factor $\eta_{n;N}^{(e_1,e_2,e_3)}$ is given by
\begin{multline*}
\eta_{n;N}^{(e_1,e_2,e_3)}=\left[\frac{(\frac{n-e_1-e_1}{2})!(n+\mu_1+\mu_2)\Gamma(\frac{n+e_1+e_2}{2}+\mu_1+\mu_2)}{2\;\Gamma(\frac{n+e_1-e_2}{2}+\mu_1+1/2)\Gamma(\frac{n+e_2-e_1}{2}+\mu_2+1/2)}\right]^{1/2}
\\
\times
\left[\frac{(N+\mu_1+\mu_2+\mu_3+1/2)(\frac{N-n-e_3}{2})!\Gamma(\frac{N+n+e_3}{2}+\mu_1+\mu_2+\mu_3+1/2)}{\Gamma(\frac{N+n-e3}{2}+\mu_1+\mu_2+1)\Gamma(\frac{N-n+e_3}{2}+\mu_3+1/2)}\right]^{1/2},
\end{multline*}
where $\Gamma(x)$ stands for the Gamma function and ensures that
\begin{align*}
\int_{0}^{2\pi}\int_{0}^{\pi}Y_{n;N}^{(e_1,e_2,e_3)}Y_{n';N'}^{(e_1',e_2',e_3')}\;h(\theta,\phi)\;\sin\theta\;\mathrm{d}\theta\mathrm{d}\phi=\delta_{nn'}\delta_{NN'}\delta_{e_1e_1'}\delta_{e_2e_2'}\delta_{e_3e_3'},
\end{align*}
where the $\mathbb{Z}_2^{3}$-invariant weight function $h(\theta,\phi)$ is \cite{Dunkl-2001}
\begin{align}
\label{weight}
h(\theta,\phi)=|\cos\theta|^{2\mu_3}|\sin\theta|^{2\mu_1}|\sin\theta|^{2\mu_2}|\cos\phi|^{2\mu_1}|\sin \phi|^{2\mu_2}.
\end{align}
\subsection{$S^2$ basis functions for BI representations}
Let $\mathcal{Y}_{K}^{N}(\theta,\phi)$, $K=0,\ldots,N$ be the functions on $S^2$ satisfying
\begin{subequations}
\label{System}
\begin{align}
\label{A-1}
\Omega\,\mathcal{Y}_{K}^{N}(\theta,\phi)&=-(N+\mu_1+\mu_2+\mu_3+1)\,\mathcal{Y}_{K}^{N}(\theta,\phi),
\\
\label{A-2}
R_1R_2R_3\,\mathcal{Y}_{K}^{N}(\theta,\phi)&=(-1)^{N}\,\mathcal{Y}_{K}^{N}(\theta,\phi),
\\
\label{C-1}
K_3\,\mathcal{Y}_{K}^{N}(\theta,\phi)&=(-1)^{K}(K+\mu_1+\mu_2+1/2)\,\mathcal{Y}_{K}^{N}(\theta,\phi).
\end{align}
\end{subequations}
where $\Omega$ is given by \eqref{Def-Omega} and where $K_3$ is given by \eqref{K-Def}. In spherical coordinates \eqref{Spherical-Coordinates}, the operator $K_3$ has the expression
\begin{align*}
K_3=\partial_{\phi}R_1+\mu_1\tan\phi(1-R_1)+\frac{\mu_2}{\tan\phi}(R_1-R_1R_2)+\mu_1 R_2+\mu_2R_1+\frac{1}{2}R_1R_2.
\end{align*}
Since $K_3$ acts only on $\phi$, the functions $\mathcal{Y}_{K}^{N}(\theta,\phi)$ can be separated.

The solutions for the azimuthal part are readily obtained from \eqref{C-1} by considering separately the eigenvalue sectors of $R_1R_2$, which commutes with $K_3$. For the positive eigenvalue sector, one finds for $K=2k+p$
\begin{subequations}
\label{Part-1}
\begin{multline}
\mathcal{F}^{(+)}_{K}(\phi)=\zeta_{K}^{(+)} \Bigg\{\left[\frac{k+1}{k+\mu_1+\mu_2+1}\right]^{p/2}\,P_{k+p}^{(\mu_2-1/2,\mu_1-1/2)}(\cos 2\phi)
\\
-(-1)^{p}\left[\frac{k+\mu_1+\mu_2+1}{k+1}\right]^{p/2}\;\cos\phi\sin\phi\;P_{k+p-1}^{(\mu_2+1/2,\mu_1+1/2)}(\cos 2\phi)\Bigg\},
\end{multline}
where $p=0,1$. For the negative eigenvalue sector, the result for $K=2k+p$ is
\begin{multline}
\mathcal{F}^{(-)}_{K}(\phi)=\zeta_{K}^{(-)} \Bigg\{\left[\frac{k+\mu_1+1/2}{k+\mu_2+1/2}\right]^{p/2}\,\sin \phi\,P_{k}^{(\mu_2+1/2,\mu_1-1/2)}(\cos 2\phi)
\\
+(-1)^{p}\left[\frac{k+\mu_2+1/2}{k+\mu_1+1/2}\right]^{p/2}\;\cos\phi\;P_{k}^{(\mu_2-1/2,\mu_1+1/2)}(\cos 2\phi)\Bigg\}.
\end{multline}
\end{subequations}
The normalization factors are
\begin{subequations}
\begin{align*}
\zeta_{K}^{(+)}&=\sqrt{\frac{(k+p)!\,\Gamma(k+\mu_1+\mu_2+1+p)}{2\,\Gamma(k+\mu_1+1/2+p)\Gamma(k+\mu_2+1/2+p)}},
\\
\zeta_{K}^{(-)}&=\sqrt{\frac{k!\,\Gamma(k+\mu_1+\mu_2+1)}{2\,\Gamma(k+\mu_1+1/2)\Gamma(k+\mu_2+1/2)}},
\end{align*}
\end{subequations}
Using \eqref{Part-1} the remaining equations \eqref{A-1}, \eqref{A-2} can be solved. When $N=2n$ and $K=2k+p$, one finds
\begin{subequations}
\label{Part-2}
\begin{multline}
\label{N-even}
\mathcal{Y}_{K}^{N}(\theta,\phi)=\sqrt{\frac{(n-k-p)!\,\Gamma(n+k+\mu_1+\mu_2+\mu_3+3/2)}{\Gamma(n+k+\mu_1+\mu_2+1)\Gamma(n-k+\mu_3+1/2-p)}}\;\times
\\
\Bigg\{\left[\frac{n-k+\mu_3-1/2}{n+k+\mu_1+\mu_2+1}\right]^{p/2}\,\sin^{2k+2p}\theta\,P_{n-k-p}^{(2k+2p+\mu_1+\mu_2,\mu_3-1/2)}(\cos 2\theta)\,\mathcal{F}^{(+)}_{K}(\phi)
\\
+\left[\frac{n+k+\mu_1+\mu_2+1}{n-k+\mu_3-1/2}\right]^{p/2}\,\cos\theta\sin^{2k+1}\theta\,P_{n-k-1}^{(2k+1+\mu_1+\mu_2,\mu_3+1/2)}(\cos 2\theta)\,\mathcal{F}^{(-)}_{K}(\phi)\Bigg\}.
\end{multline}
When $N=2n+1$ and $K=2k+p$, the result is
\begin{multline}
\label{N-odd}
\mathcal{Y}_{K}^{N}(\theta,\phi)=(-1)^{K}\sqrt{\frac{(n-k)!\Gamma(n+k+\mu_1+\mu_2+\mu_3+3/2+p)}{\Gamma(n-k+\mu_3+1/2)\Gamma(n+k+\mu_1+\mu_2+1+p)}}\;\times
\\
 \Bigg\{\left[\frac{n+k+\mu_1+\mu_2+1}{n-k+\mu_3+1/2}\right]^{(1-p)/2} \hspace{-.99cm} \cos \theta \sin^{2k+2p}\theta\;P_{n-k-p}^{(2k+2p+\mu_1+\mu_2,\mu_3+1/2)}(\cos 2\theta)\,\mathcal{F}^{(+)}_{K}(\phi)
\\
-\left[\frac{n-k+\mu_3+1/2}{n+k+\mu_1+\mu_2+1}\right]^{(1-p)/2} \hspace{-.5cm}\sin^{2k+1}\theta\;P_{n-k}^{(2k+1+\mu_1+\mu_2,\mu_3-1/2)}(\cos 2\theta)\,\mathcal{F}^{(-)}_K(\phi)\Bigg\}.
\end{multline}
\end{subequations}
The solutions to \eqref{System} can be expressed as linear combinations of the Dunkl spherical harmonics \eqref{Dunkl-Harmonics}. For $N=2n$, straightforward calculations lead to the expressions
\begin{align*}
&\mathcal{Y}_{2k}^{N}(\theta,\phi)=
\sqrt{\frac{n+k+\mu_1+\mu_2+\mu_3+1/2}{2n+\mu_1+\mu_2+\mu_3+1/2}}
\Bigg\{
\sqrt{\frac{k+\mu_1+\mu_2}{2k+\mu_1+\mu_2}}Y_{2k;N}^{(0,0,0)}(\theta,\phi)
\\
&-\sqrt{\frac{k}{2k+\mu_1+\mu_2}}Y_{2k;N}^{(1,1,0)}(\theta,\phi)
\Bigg\}
+
\sqrt{\frac{n-k}{2n+\mu_1+\mu_2+\mu_3+1/2}}\quad \times
\\
&\Bigg\{
\sqrt{\frac{k+\mu_2+1/2}{2k+\mu_1+\mu_2+1}}Y_{2k+1;N}^{(0,1,1)}(\theta,\phi)
+
\sqrt{\frac{k+\mu_1+1/2}{2k+\mu_1+\mu_2+1}}Y_{2k+1;N}^{(1,0,1)}(\theta,\phi)\Bigg\},
\\[0.3cm]
&\mathcal{Y}^{N}_{2k+1}(\theta,\phi)=\sqrt{\frac{n-k+\mu_3-1/2}{2n+\mu_1+\mu_2+\mu_3+1/2}}
\Bigg\{\sqrt{\frac{k+1}{2k+\mu_1+\mu_2+2}} Y_{2k+2;N}^{(0,0,0)}(\theta,\phi)
\\
&+\sqrt{\frac{k+\mu_1+\mu_2+1}{2k+\mu_1+\mu_2+2}}Y_{2k+2;N}^{(0,1,1)}(\theta,\phi)\Bigg\}
+\sqrt{\frac{n+k+\mu_1+\mu_2+1}{2n+\mu_1+\mu_2+\mu_3+1/2}}\quad \times
\\
&\Bigg\{
\sqrt{\frac{k+\mu_1+1/2}{2k+\mu_1+\mu_2+1}}Y_{2k+1;N}^{(0,1,1)}(\theta,\phi)-\sqrt{\frac{k+\mu_2+1/2}{2k+\mu_1+\mu_2+1}}Y_{2k+1;N}^{(1,0,1)}(\theta,\phi)\Bigg\},
\end{align*}
For $N=2n+1$, one finds
\begin{align*}
&\mathcal{Y}_{2k}^{N}(\theta,\phi)=\sqrt{\frac{k+n+\mu_1+\mu_2+1}{2n+\mu_1+\mu_2+\mu_3+3/2}}\Bigg\{
\sqrt{\frac{k+\mu_1+\mu_2}{2k+\mu_1+\mu_2}}Y_{2k;N}^{(0,0,1)}(\theta,\phi)
\\
&-\sqrt{\frac{k}{2k+\mu_1+\mu_2}}Y_{2k;N}^{(1,1,1)}(\theta,\phi)\Bigg\}-\sqrt{\frac{n-k+\mu_3+1/2}{2n+\mu_1+\mu_2+\mu_3+3/2}}\quad \times
\\
&
\Bigg\{
\sqrt{\frac{k+\mu_2+1/2}{2k+\mu_1+\mu_2+1}}Y_{2k+1;N}^{(0,1,0)}(\theta,\phi)+\sqrt{\frac{k+\mu_1+1/2}{2k+\mu_1+\mu_2+1}} Y_{2k+1;N}^{(1,0,0)}(\theta,\phi)\Bigg\},
\end{align*}
\begin{align*}
&\mathcal{Y}_{2k+1}^{N}(\theta,\phi)=\sqrt{\frac{n+k+\mu_1+\mu_2+\mu_3+3/2}{2n+\mu_1+\mu_2+\mu_3+3/2}}\Bigg\{\sqrt{\frac{k+\mu_1+1/2}{2k+\mu_1+\mu_2+1}}Y_{2k+1;N}^{(0,1,0)}(\theta,\phi)
\\
&-\sqrt{\frac{k+\mu_2+1/2}{2k+\mu_1+\mu_2+1}}Y_{2k+1;N}^{(1,0,0)}(\theta,\phi)\Bigg\}-\sqrt{\frac{n-k}{2n+\mu_1+\mu_2+\mu_3+3/2}}\quad \times
\\
&\Bigg\{
\sqrt{\frac{k+1}{2k+\mu_1+\mu_2+2}}Y_{2k+2;N}^{(0,0,1)}(\theta,\phi)+\sqrt{\frac{k+\mu_1+\mu_2+1}{2k+\mu_1+\mu_2+2}}Y_{2k+2;N}^{(1,1,1)}(\theta,\phi)\Bigg\}.
\end{align*}
It follows from the orthogonality relation for the Jacobi polynomials \cite{Koekoek-2010} that
\begin{align}
\label{Ortho-Y}
\int_{0}^{\pi}\int_{0}^{2\pi} \mathcal{Y}_{K}^{N}(\theta,\phi)\mathcal{Y}_{K'}^{N'}(\theta,\phi)\,h(\theta,\phi)\,\sin\theta\;\mathrm{d}\phi\,\mathrm{d}\theta=\delta_{KK'}\delta_{NN'},
\end{align}
where $h(\theta,\phi)$ is given by \eqref{weight}.
\begin{proposition}
The functions $\mathcal{Y}_{K}^{N}(\theta,\phi)$ defined by \eqref{Part-1}, \eqref{Part-2} realize the Bannai--Ito modules of Proposition 4. That is, if one takes $\psi_{K}=\mathcal{Y}^{N}_{K}(\theta,\phi)$, the generators \eqref{K-Def} expressed in spherical coordinates have the actions \eqref{Actions-4}.
\end{proposition}
\begin{proof}
The result follows from the fact that the $\mathcal{Y}_{K}^{N}(\theta,\phi)$ are solutions to \eqref{System}. One needs only to check for possible phase factors. A check on the highest order term occurring in $K_1\mathcal{Y}^{N}_{K}(\theta,\phi)$ confirms the phase factors in \eqref{Part-2}.
\end{proof}
\subsection{Bannai--Ito polynomials as overlap coefficients} As is seen from \eqref{System}, the simultaneous diagonalization of $\Omega$, $R_1R_2R_3$ and $K_3$ is associated to the separation of variables of the basis functions $\mathcal{Y}^{N}_{K}(\theta,\phi)$ in the usual spherical coordinates
\begin{align}
\label{Spherical-Coordinates-1-1}
x_1=\sin \theta \cos \phi,\quad x_2=\sin\theta \sin \phi,\quad x_3=\cos\theta.
\end{align}
Consider the basis functions $\mathcal{Z}^{N}_{S}(\vartheta,\varphi)$, $S=0,\ldots,N$, associated to the simultaneous diagonalization of $\Omega$, $R_1R_2R_3$ and $K_1$. The relations \eqref{A-1}, \eqref{A-2} hold and one has
\begin{align}
\label{Ultra-2}
K_1\mathcal{Z}^{N}_{S}(\vartheta,\varphi)=(-1)^{S}(S+\mu_2+\mu_3+1/2)\mathcal{Z}_{S}^{N}(\vartheta,\varphi).
\end{align}
The functions $\mathcal{Z}^{N}_{S}(\vartheta,\varphi)$ separate in the alternative spherical coordinates
\begin{align}
\label{Spherical-Coordinates-2}
x_1=\cos \vartheta,\quad x_2=\sin \vartheta\cos \varphi,\quad x_3=\sin \vartheta\sin \varphi,
\end{align}
as can be seen from the expression of $K_1$ obtained using \eqref{Spherical-Coordinates-2}. Writing $\Omega$ in the coordinates \eqref{Spherical-Coordinates-2} and comparing the expression with the one obtained using the coordinates \eqref{Spherical-Coordinates-1-1}, it is seen that the basis functions $\mathcal{Z}_{S}^{N}(\vartheta,\varphi)$ have the expression
\begin{align*}
\mathcal{Z}_{S}^{N}(\vartheta,\varphi)=
\begin{cases}
\pi \mathcal{Y}_{S}^{N}(\pi-\vartheta,\varphi), & \text{$N$ is even},
\\
\pi \mathcal{Y}_{S}^{N}(\vartheta,\varphi), & \text{$N$ is odd},
\end{cases}
\end{align*}
where $\pi=(123)$ is the permutation applied to the parameters $(\mu_1,\mu_2,\mu_3)$. Since $\{\mathcal{Y}_{K}^{N}(\theta,\phi)\}_{K=0}^{N}$ and $\{\mathcal{Z}_{S}^{N}(\vartheta,\varphi)\}_{S=0}^{N}$ form orthonormal bases for the same space, they are related (at a given point) by a unitary transformation. One hence writes
\begin{align}
\label{Overlaps}
\mathcal{Z}_{S}^{N}(\vartheta,\varphi)=\sum_{K=0}^{N}R^{\mu_1,\mu_2,\mu_3}_{S,K;N}\,\mathcal{Y}_{K}^{N}(\theta,\phi).
\end{align}
Since the coefficients $R^{\mu_1,\mu_2,\mu_3}_{S,K;N}$ are real, their unitarity implies
\begin{align}
\label{Unitarity}
\sum_{S=0}^{N}R^{\mu_1\mu_2\mu_3}_{S,K;N}R^{\mu_1\mu_2\mu_3}_{S,K';N}=\delta_{KK'},\;\;\sum_{K=0}^{N}R^{\mu_1\mu_2\mu_3}_{S,K;N}R^{\mu_1\mu_2\mu_3}_{S',K;N}=\delta_{SS'},
\end{align}
These transition coefficients can be expressed in terms of the Bannai--Ito polynomials \eqref{BI-Recurrence} as follows. Acting with $K_1$ on both sides of \eqref{Overlaps}, using \eqref{Ultra-2} and Proposition 5 and furthermore defining $R^{\mu_1,\mu_2,\mu_3}_{S,K;N}=2^{K}[w_{S;N}]^{1/2}B_{K}(x_{S})$ such that $B_{0}(x_{S})=1$, it seen that $B_{K}(x_{S})$ satisfy the three-term recurrence relation \eqref{BI-Recurrence} of the Bannai--Ito polynomials $B_{K}(x_{S};\rho_1,\rho_2,r_1,r_2)$ with parameters
\begin{align}
\label{Parameters}
\rho_1=\frac{\mu_2+\mu_3}{2},\;\rho_2=\frac{\mu_1+\mu}{2},\; r_1=\frac{\mu_3-\mu_2}{2},\;r_2=\frac{\mu-\mu_1}{2}.
\end{align}
with $\mu$ given by \eqref{Mu} and where the variable $x_{S}$ is given by
\begin{align}
\label{Variable}
x_{S}=\frac{1}{2}\left[(-1)^{S}(S+\mu_2+\mu_3+1/2)-1/2\right].
\end{align}
The coefficients $R^{\mu_1\mu_2\mu_3}_{S,K;N}$ coincide with the Racah coefficients of $sl_{-1}(2)$ \cite{Genest-2012}. Combining \eqref{Unitarity} with the orthogonality relation of the BI polynomial \cite{Tsujimoto-2012-03}, one finds
\begin{align}
\label{Racah-Coef}
R^{\mu_1\mu_2\mu_3}_{S,K;N}=\sqrt{\frac{w_{S;N}}{u_1u_{2}\cdots u_{K}}}B_{K}(x_{S};\rho_1,\rho_2,r_1,r_2).
\end{align}
with \eqref{Parameters}, \eqref{Variable}, where $u_{n}=A_{n-1}C_{n}$ with $A_{n}$, $C_{n}$ as in \eqref{Recu-Coeff}, and where $w_{S;N}$ is of the form
\begin{align*}
w_{S;N}=\frac{1}{h_{N}}\frac{(-1)^{\nu}(\rho_1-r_1+1/2;\rho_1-r_2+1/2)_{\ell+\nu}(\rho_1+\rho_2+1;2\rho_1+1)_{\ell}}{(\rho_1+r_1+1/2;\rho_1+r_2+1/2)_{\ell+\nu}(1;\rho_1-\rho_2+1)_{\ell}},
\end{align*}
with $S=2\ell+\nu$, $\nu=\{0,1\}$ and 
\begin{align}
(a_1;a_2;\ldots;a_{k})_{n}=(a_1)_{n}(a_2)_{n}\cdots (a_k)_{n},
\end{align}
where $(a)_n=a(a+1)\cdots(a+n-1)$. The normalization factor $h_{N}$ is given by
\begin{align*}
h_{N}=
\begin{cases}
\frac{(2\rho_1+1;r_1-\rho_2+1/2)_{N/2}}{(\rho_1-\rho_2+1;\rho_1+r_1+1/2)_{N/2}}, & \text{$N$ even},
\\
\frac{(2\rho_1+1;r_1+r_2)_{(N+1)/2}}{(\rho_1+r_1+1/2;\rho_1+r_2+1/2)_{(N+1)/2}}, & \text{$N$ odd}.
\end{cases}
\end{align*}
Using the orthogonality relation \eqref{Ortho-Y} satisfied by the basis functions $\mathcal{Y}_{K}^{N}(\theta,\phi)$ on the decomposition formula \eqref{Overlaps}, one finds that
\begin{align*}
R^{\mu_1\mu_2\mu_3}_{S,K;N}=\int_{0}^{\pi}\int_{0}^{2\pi}\mathcal{Y}_{K}^{N}(\theta,\phi)\,\mathcal{Z}_{S}^{N}(\vartheta,\varphi)\,h(\theta,\phi)\,\sin\theta\,\mathrm{d}\phi\,\mathrm{d}\theta,
\end{align*}
which in light of \eqref{Racah-Coef} gives an integral formula for the Bannai--Ito polynomials.
\section{Conclusion}
We have established in this paper the algebraic basis for the harmonic analysis on $S^2$ associated to a $\mathbb{Z}_2^{3}$ Dunkl Laplacian $\Delta_{S}^2$. The commutant of $\Delta_{S^2}$ was determined in the framework of the Racah problem for $sl_{-1}(2)$ and identified with a central extension of the Bannai--Ito algebra. Two bases for the unitary irreducible representations of this algebra on $L^2{(S^2)}$ were explicitly constructed in terms of the Dunkl spherical harmonics with the Bannai--Ito orthogonal polynomials arising in their overlaps.

Since the Dunkl operators and Laplacian can be defined for an arbitrary number of variables, it would be natural to look for the extension of the results presented here to spheres in higher dimensions. It would be also of interest to examine the situation on hyperboloids. We plan to pursue the study of these questions in the future.

\section*{Acknowledgements}
\noindent
V.X.G. holds a fellowship from the Natural Sciences and Engineering Research Council of Canada (NSERC). The research of L.V. is supported in part by NSERC.

\bibliographystyle{amsplain}
%\bibliography{/home/vxg/Documents/References/References_VXG}
\providecommand{\bysame}{\leavevmode\hbox to3em{\hrulefill}\thinspace}
\providecommand{\MR}{\relax\ifhmode\unskip\space\fi MR }
\providecommand{\MRhref}[2]{%
  \href{http://www.ams.org/mathscinet-getitem?mr=#1}{#2}
}
\providecommand{\href}[2]{#2}

\end{document}